\documentclass{eect}
\usepackage{amsmath}
  \usepackage{paralist}
  \usepackage{graphics} 
  \usepackage{epsfig} 
 \usepackage[colorlinks=true]{hyperref}
\hypersetup{urlcolor=blue, citecolor=red}

\def\currentvolume{2}
 \def\currentissue{4}
  \def\currentyear{2013}
   \def\currentmonth{December}
    \def\ppages{669--677} 
     \def\DOI{10.3934/eect.2013.2.669}

\newtheorem{theorem}{Theorem}[section]

\theoremstyle{definition}

\newtheorem{remark}{Remark}

\newcommand{\ep}{\varepsilon}

\newcommand{\RR}{{\mathbb R}}

\newcommand{\pa}{\partial}
\renewcommand{\Box}{\square}

\newcommand{\bv}{\mathbf{v}}
\newcommand{\bw}{\mathbf{w}}
\renewcommand{\vec}[1]{{\mathbf{#1}}}

\title[Control of blow-up]
      {Control of blow-up singularities\\ for nonlinear wave equations}

\author[Satyanad Kichenassamy]{}



 \email{satyanad.kichenassamy@univ-reims.fr}

\begin{document}
\maketitle

\centerline{\scshape Satyanad Kichenassamy}
\medskip
{\footnotesize
 \centerline{Laboratoire de Math\'ematiques}
   \centerline{Universit\'e de Reims Champagne-Ardenne}
   \centerline{Moulin de la Housse, B.P.\ 1039}
   \centerline{F-51687 Reims Cedex 2, France}
} 




\begin{abstract}
While the global boundary control of nonlinear wave equations that
exhibit blow-up is generally impossible, we show on a typical
example, motivated by laser breakdown, that it is possible to
control solutions with small data so that they blow up on a
prescribed compact set bounded away from the boundary of the
domain. This is achieved using the representation of singular
solutions with prescribed blow-up surface given by Fuchsian
reduction. We outline on this example simple methods that may be
of wider applicability.
\end{abstract}

\section{Introduction}

\subsection{Objectives}

It is well-known that the boundary control of solutions of
nonlinear Klein-Gordon equations that exhibit blow-up in finite
time is, in general, impossible. Indeed, assume that the given
initial and boundary data lead to blow-up for $t=t_0$ and $x=x_0$,
where the distance of $x_0$ from the boundary is greater than
$ct_0$, $c$ being the speed of propagation; then, the boundary
data do not have time to influence the solution at $x_0$ before
the blow-up time. The solution near the blow-up point is entirely
determined by the initial conditions, and no choice of boundary
conditions can modify this blow-up behavior. In other words,
boundary conditions, whatever their type, cannot, in general,
arrest blow-up. Nevertheless, it is often possible to steer small
data to zero (``local controllability'').\footnote{General methods
and results on the control of hyperbolic problems, together with
further references, may be found in the following papers:
\cite{russell,l-t-1,l-t-2,lions-1,l,l-1,b,z-hbk,z-z}. Among works
more particularly relevant to the present paper, we may mention
\cite{ch,ci,yz-zl,f,z-jmpa,z-cdf}, where further references may be
found. We do not aim at completeness. The considerations of the
present paper also apply to other nonlinearities, such as those
considered in \cite{z-jmpa}. For recent results on interior
control, see e.g.\ \cite{d-l-z}.}  The purpose of this paper is to
show that it is also possible, for cubic wave equations, to steer
small data in order to achieve blow-up on a prescribed compact set
in the interior of the domain.

This possibility is suggested by the method of Fuchsian reduction
\cite{k-l-1,k-l-2,fr} that yields solutions that blow up on a
given set in spacetime, for wide classes of equations. More
precisely, given a sufficiently smooth graph
$\Sigma=\{(x,t)\in\RR^n\times\RR : t=\psi(x)\}$, with
$|\nabla\psi|<1$ on the entire space, the method yields a solution
$u(x,t)$ that becomes singular precisely as $t\to \psi(x)-$, and
that is defined and regular in an open set limited by $\Sigma$.
Walter Littman observed to me, many years ago, that this type of
result suggests the possibility of a control of blow-up
singularities, since it furnishes an explicit construction of
boundary data that steer a particular set of Cauchy data so that
they blow up on a specified set, while remaining smooth elsewhere.
Indeed, taking $\psi$ to be positive, and a smooth bounded domain
$\Omega\subset\RR^n$ that contains a compact set $K$ on which
$\psi$ reaches its (positive) minimum $\psi_{\text{min}}$, the
restriction of $u$ and $u_t$ to the set where $t=0$ and
$x\in\Omega$ furnishes a pair of Cauchy data, and its trace on the
boundary of $\Omega$ provides Dirichlet boundary data, such that
the solution of the initial-boundary value problem with these data
first blows up precisely for $t=\psi_{\text{min}}$ and $x\in K$.
In other words,
\begin{quote}
\emph{It may not be possible to arrest blow-up, but it may be possible to force blow-up to occur at a specified time and place.}
\end{quote}
This explicit construction has the same advantages as the
classical restriction argument for the linear wave equation
\cite{russell-2,l-1}. However, it requires the initial data, as
well as the boundary data, to be chosen in a special way to ensure
blow-up occurs only on $\Sigma$. We show in this paper, on a
typical example that may be of some interest in applications, that
if the problem is locally controllable, it is possible to steer
the solution, starting from arbitrary small data, so that it blows
up on a prescribed set inside the domain. In a nutshell, the above
construction will be modified so as to ensure that $u$ has not
only smooth, but also \emph{small} data, that may, in turn, be
steered to zero by local controllability.

\subsection{The model}

Our model is
\begin{equation}\label{weq} \Box u = 2u^3,
\end{equation}
in three space dimensions, to fix ideas. Similar considerations
apply to complex solutions of $\Box u +\alpha u_z = \beta u|u|^2$,
where $\alpha$ and $\beta$ are constants. This latter problem is a
model for the envelope of the electric field of an ultra-short
optical pulse, taking normal dispersion and paraxiality
corrections into account. In the language of laser breakdown, our
statement may be translated as follows:
\begin{quote}
\emph{While  it is impossible to arrest laser self-focusing, it is possible to arrange boundary data so that breakdown occurs at a place and time specified in advance.}
\end{quote}
Before outlining the proof of this statement, we introduce some
notation. We shall have to work with two sets of variables: the
original space and time variables $(x,t)\in\RR^n\times\RR$, and
variables adapted to $\Sigma$: $T=\psi(x)-t$; $X=x$. We write
$\pa_i$ for $\pa/\pa X_i$, where $i$ runs from 1 to $n$; we have
$n=3$ in the example from nonlinear optics, but this will not be
used in the sequel. Also, $\psi$ may be viewed as a function of
$x$ or $X$. The smoothness of $\Sigma$ is measured in Sobolev
spaces: we take $\psi\in H^\sigma$, where $\sigma$ will be taken
sufficiently large. Throughout, we assume $\sup|\nabla\psi|<1$ and
$\sup|\psi|<1$ everywhere. Therefore, $|t-T|<1$. The regularity of
the solution $u(X,T)$ will be estimated in $H^s\times H^{s-1}$,
with $s\geq\sigma-4$, this choice being dictated by the regularity
of the coefficients of the Fuchsian system introduced in Sect.~3.
We also let $S=(1-\Delta)^{s/2}$. Finally, $\|\quad\|_s$ stands
for the norm in $H^s$. Recall that multiplication is a continuous
bilinear map from $H^s\times H^s$ to $H^s$.

\subsection{Outline of the argument}

The argument for proving this controllability of blow-up
singularities for the model at hand is as follows. We are given
the compact set $K$ within $\Omega$. We are also given $s_0$ and
$\ep>0$ so that one has local controllability in the smooth
bounded domain $\Omega$ for Cauchy data of norm less than $\ep$ in
$H^{s_0}\times H^{s_0-1}(\Omega)$ \cite{yz-zl,ch}. We may assume
$s_0>n/2+1$ without loss of generality. We first of all choose
$\alpha>2$ so that the Cauchy data for the exact solution $1/t$,
for $t=\alpha$, have norm less than $\ep/4$ in $H^{s_0}\times
H^{s_0-1}$. The objective is to show that, if the constants $\alpha$
and $\sigma$ are taken large enough, there is a constant $\mu$
such that $\|\psi\|_\sigma<\mu$ ensures that there are solutions
$u$ of (\ref{weq}) that blow-up for $t=\psi(x)$ and have data on
the hyperplane $(t=\alpha)$ that are less than $\ep$ in
$H^{s_0}\times H^{s_0-1}(\Omega)$. It is always possible to choose
$\psi$ so that it vanishes precisely on $K$, and is negative
elsewhere; one may also assume its $H^\sigma$ norm to be as small
as we wish---consider $\lambda\psi$, with $\lambda$ positive and
small if necessary. By time reversal (considering $u(\alpha-t)$),
we obtain a solution with Cauchy data on $(t=0)$ that are less
than $\ep$ in norm, and that first blows up on $K$. Taking the
trace of this solution on $\pa\Omega$, we obtain the result
\begin{quote}
There are small Cauchy data, and boundary controls on $\pa\Omega$ that yield a solution that blows up for $t=\alpha$ and $x\in K$, and remains finite for $t=\alpha$ and $x\in \Omega\setminus K$.
\end{quote}
Combining this with the local controllability result gives the
desired boundary control of blow-up singularities.

The rest of the paper is devoted to showing that one can choose
$\alpha$, $\sigma$ and $\mu$ with the above properties. This is
achieved by constructing a solution $u$ of (\ref{weq}) consisting
of three parts:
\begin{equation}
u=\frac1t+\Phi+T^3w,
\end{equation}
where $T=t-\psi(x)$, and $\Phi$ is an explicit expression
involving $\psi$ and its derivatives, and that vanishes when
$\psi$ is identically zero. In fact, $1/t$ is an exact solution of
(\ref{weq}). In Sect.~2, $\Phi$ is obtained by truncating a formal
solution of (\ref{weq}); it is completely determined by $\psi$,
and has small Cauchy data on $(t=\alpha)$ is $\psi$ is small. In
Sect.~3, $w$ is found as the solution of a degenerate
initial-value problem of Fuchsian type. The restriction $w_0$ of
$w$ to $\Sigma$, must be specified in order to determine $w$.
Sect.~4 deals with the estimation of $w$. Since $|t-T|<1$, in
order to estimate the Cauchy data of $T^3w$ in $H^{s_0}\times
H^{s_0-1}$ on the hyperplane $\{t=\alpha\}$, it suffices  to
estimate the \emph{space-time} Sobolev norm of index $s$ of $w$ on
some slab of the form $(\alpha-1\leq T\leq b)$, where
$b>\alpha+1$, with $s>s_0+1/2$, and then take the traces of $T^3w$
and $(T^3w)_t$ on the hyperplane $(t=\alpha)$. It is such an
estimate that we obtain in Sect.~5: these traces are small if
$\|\psi\|_\sigma+\|w_0\|_s$ is, provided $\sigma$ is large enough.
The estimation of $u$ is then easily completed: $\alpha$ has been
chosen at the outset to make the Cauchy data of $1/t$ less than
$\ep/4$; the Cauchy data of $\Phi$ have the same property if
$\|\psi\|_\sigma$ is less than some $M_1$, and those of $T^3w$ are
less than $\ep/2$ if $\|\psi\|_\sigma+\|w_0\|_s$ does not exceed
some $M_2$. Therefore, if $\|\psi\|_\sigma<\min(M_1,M_2)$,
and $\psi$ is nonpositive and vanishes only on $K$, we may take
$w_0$ so that the resulting solution $u$ has Cauchy data less than
$\ep$ in norm, and blows up precisely on the compact $K$, as
desired .

\section{Step 1: Introduction of the formal expansion of $u$}

In the variables $(X,T)$, the wave equation (\ref{weq}) takes the form
\begin{equation}\label{weqT}
\gamma u_{TT}-\Delta_X u+2\nabla\psi\cdot\nabla u_T+u_T\Delta\psi=2u^3,
\end{equation}
where $\gamma=1-|\nabla\psi|^2$. The solvability of the standard
Cauchy problem for (\ref{weqT}) with Cauchy data $(f,g)$ on
$\Sigma$ means that, if $f$ and $g$ in suitable function spaces,
there is, near $\Sigma$, a unique function $v$ such that
$u=f+Tg+T^2v$ solves (\ref{weqT}).\footnote{The local solvability
of this problem is a very special case of standard results on
symmetric-hyperbolic systems, since any strictly hyperbolic
operator admits of symmetrization, see e.g.~\cite[\S
5.2--5.3]{taylor}. An explicit reduction is given below, see
(\ref{eq:sym_syst_cubic}).} Singular solutions may be obtained by
a closely related Ansatz: seek solutions in the form
\[ u=\frac 1T\left\{u_0+u_1T+\dots\right\}.
\]
This may be viewed as a perturbation of the exact solution $1/t$ of (\ref{weq}).\footnote{The existence of an exact solution simplifies matters, but is not essential: it is possible to construct singular solutions by a similar Ansatz even if there is no exact solution independent of space variables. Also, since $-u$ is a solution if $u$ is, there are also solutions that blow up to $-\infty$. Both solutions are very easy to produce numerically, by using a local explicit scheme \cite{cabart}.} In the present situation, (\ref{weq}) is formally solved by an expression of the form
\[ u=\frac 1T\left\{u_0+u_1T+u_2T^2+u_3T^3+u_{4,1}T^4\ln T+T^4w\right\},
\]
where $w$ is a series in $T$ and $T\ln T$, with coefficients
depending on $X$.\footnote{It is convenient to treat $T$ and $T\ln
T$ as if they were independent variables for bookkeeping purposes,
when computing formal series solutions.} The coefficients $u_0$,
$u_1$, $u_2$, $u_3$ and $u_{4,1}$ are entirely determined by
$\psi$ and its derivatives up to order four; one finds
$u_0=\sqrt\gamma$. They may be found recursively, and have a
geometric interpretation \cite[pp.~271--273]{fr}, \cite{cabart}.
All we need here is that they are obtained by dividing polynomials
in derivatives of $\psi$ by powers of $\gamma$ and that, apart
from $u_0$, they vanish when $\psi$ is identically zero. The rest
of the series is entirely determined by the value $w_0(X)$ of $w$
for $T=0$, and the coefficients of the expansion of $w$ may be
found recursively, as long as $\psi$ possesses sufficiently many
derivatives. Thus, the Cauchy data are replaced by the pair of
\emph{singularity data} $(\psi,w_{0})$.

There are two essential differences with the Cauchy problem: the
expansion must include logarithmic terms, even if the solution is
infinitely smooth off $\Sigma$, and the singularity data are not
the first two terms in the expansion.\footnote{There are general
rules to determine the form of the expansion and the nature of the
data (see \cite{fr}), that generalize the usual rules for the form
of series solutions of ODEs of Fuchsian type, such as the Bessel
or hypergeometric equations, hence the name of the method;
however, the present solutions are not necessarily analytic. More
general examples require even more complicated series solutions
that have no counterpart in the ODE case.} For our purposes, the
exact expression for the coefficients of the expansion is not
needed. It suffices to write
\[ u = \frac 1t + \Phi + T^3w,
\]
where
\[ \Phi=\left(\frac {u_0}T-\frac 1t\right)+u_1+u_2T+u_3T^2+u_{4,1}T^3\ln T.
\]
Since $\alpha>2$, this expression has no singularity on $(t=\alpha)$; it also vanishes with $\psi$ (and its derivatives).  Because smooth functions act on Sobolev spaces of index higher than $n/2$, is follows that
\[ \|\Phi\|_{s}+ \|\Phi_t\|_{s-1}\leq C(\alpha)\|\psi\|_\sigma(1+\|\psi\|_\sigma^q)
\]
for a suitable integer $q$, and the Sobolev norms are taken on the
hyperplane $t=\alpha$. By the above estimate on $\Phi$, there is a
constant $M_1$, that also depends on $\alpha$, hence on $\ep$,
such that the Cauchy data of $\Phi$ have the same property if
$\|\psi\|_\sigma<M_1$.

\section{Step 2: Reduced equation for $w$.}

Let us examine the equation satisfied by $w$: this is the reduced
equation (RE). Writing $D=T\pa_T$, its form is
\[ (\gamma D(D+5)-T^2\Delta_X+T^2\nabla\psi\nabla\pa_T) w = \text{lower-order terms},
\]
where $\gamma=1-|\nabla \psi (x)|^2 $. Observe that the singular
set $(T=0)$ is characteristic for the operator on the left-hand
side, but not for the wave operator. To obtain $w$, we solve the
initial-value problem for the RE with only one initial condition:
$w(T=0)=w_{0}$. This may be achieved by casting the problem in the
form of a first-order reduced system (RS), that is symmetric
hyperbolic for $T\neq 0$ \cite[Th.~10.10, pp.~186-8]{fr}:
\begin{theorem}\label{th:sym_pb}
For $s$ and $m$ large enough, there are symmetric matrices $Q$ and $A^j$, $1\leq j\leq n$, a constant matrix $A$ and functions $f_0$ and $f_1$ such that the solution $\vec{w}=(w,w_{(0)},w_{(i)})$ of the reduced system (RS)
\begin{equation}\label{eq:symmetric}
Q(D+A) \vec{w}=T A^j \pa_{j} \vec{w}+T f_0(T,T\ln T,X,\vec{w})+T\ln T f_1(T,T\ln T,X,\vec{w}),
\end{equation}
exists for small $T$, and generates a solution $u=1/t+\Phi+T^3 w$ of the wave equation,
provided that $\bw(T=0)$ is small in $H^s$ and belongs to the null-space of $A$.
\end{theorem}
\begin{proof}
The RS is derived from the usual symmetric system associated with the wave equation (\ref{weq}) in the new variables
$X$ and $T$; letting $\vec{u}:=(u,u_{(0)},u_{(i)})$, where $u_{(0)}$ and the $u_{(i)}$ correspond to the time and space derivatives of $u$, this system reads:
\begin{equation}\label{eq:sym_syst_cubic}
\left\{
\begin{aligned}
\pa_T u &=u_{(0)},\\ (1-|\nabla\psi|^2) \pa_T u_{(0)} &=\sum_i
(\pa_{i} u_{(i)}-2 \psi_i \pa_{i} u_{(0)}) - (\Delta \psi)
u_{(0)}+2 u^3,\\ \pa_T u_{(i)} &=\pa_{i} u_{(0)}.
\end{aligned}
\right.
\end{equation}
Define the unknown $\bw=(w,w_{(0)},w_{(i)})$ through
\begin{equation}
\left\{
\begin{aligned}
u&=\frac{u_0}{t_0}+u_1+u_2 t_0+u_3 t_0^2 +u_{4,1} t_0^2
t_1+w(t_0,t_1,X) t_0^3,\\ u_{(0)}&=-\frac{u_0}{t_0^2}+u_2+2u_3
t_0+(t_0^2+3 t_0 t_1) u_{4,1} +w_{(0)}(t_0,t_1,X) t_0^2,\\
u_{(i)}&=\frac{u_{0i}}{t_0}+u_{1i}+u_{2i} t_0+u_{3i}
t_0^2+w_{(i)}(t_0,t_1,X) t_0^2.
\end{aligned}
\right.
\end{equation}
About the derivation of this expression, see Remark 1 below. After substitution, the symmetric system for $(u,u_{(0)},u_{(i)})$ goes into the desired system, where $Q,\,A^j,\,A$ are given by
\[
Q=\left[
\begin{array}{ccc}
1 &                       & \\
  & \gamma                & \\
  &                       &I_n\\

\end{array}
\right],\quad \text{and }\, A^j=\left[
\begin{array}{ccc}
0 & 0          &    0  \\ 0 & -2 \psi^j  &  e^j \\ 0 &  ^{t}e^j
&   0  \\
\end{array}
\right],
\]
where $e^j=(0,\dots,1,\dots,0)$ is the $j$th vector of the standard basis of $n$-space, and
\[
 A=\left[
\begin{array}{cc|ccc}
     3    &   -1  & 0 &\cdots& 0\\
    -6    &    2  & 0 &\cdots& 0\\
    \hline
     0    &    0  & 2 &      & 0\\
    \vdots& \vdots&   &\ddots&  \\
     0    &    0  & 0 &      & 2
    \end{array}
\right].
\]
The matrix $A$ is constant, with eigenvalues 0 and 5; the former is simple. For $T=0$, (\ref{eq:symmetric}) forces $Aw(T=0)=0$, which is why the solution of the RS is determined by only one initial value, namely, the first component of $\vec w$.
The exact form of $f=Tf_0+T\ln T f_1$ is again not essential:\footnote{See \cite[p.~187]{fr} for the complete expressions.} all we need is that
\[ f=g_0+g_1w+g_2w^2+g_3w^3,
\]
where the coefficients $g_k$ are polynomials in $T$ and $T\ln T$ without constant term, with coefficients that are products of the coefficients of $\Phi$ and of $\gamma$. The coefficient $g_0$ vanishes with $\psi$.
\end{proof}
\begin{remark}
The RS was derived by the following argument: to obtain the reduced first-order system corresponding to a nonlinear wave equation, first reduce it to a symmetric-hyperbolic system for the unknown $u$ and its first derivatives $u_k$ ($0\leq k\leq n$). Determine the formal expansion $a_h$ of $u$ up to some given order $h$ inclusive, and the expansion $a_{h-1,k}$ of $u_k$ up to order $h-1$. Then, let $v=(u-a_h)/T^h$, $v_k=(u_k-a_{h-1,k})/T^{h-1}$. The resulting system for $v$ and the $v_k$ is the desired reduced system if $h$ is sufficiently large.
\end{remark}
\begin{remark}
It may be shown that the reduced system has a unique local solution that may be viewed as a continuous function of $T$ and $T\ln T$ with values in a Sobolev space, the existence proof being carried out by performing the same computations on a regularized system obtained by Yosida regularization, or using Friedrichs mollifiers, as in the symmetric-hyperbolic case, see e.g.\ \cite{taylor,fr}.
\end{remark}
\begin{remark}
Even though $Q$ is positive definite, and $A$ has no eigenvalues with negative real parts, $A$ is not positive definite. For this reason, we shall introduce in the next section a weighted the $L^2$ scalar product. This annoyance could have been avoided at the expense of further expansion of the solution by introducing a new unknown $\mathbf z=[\bw-\bw_0-T\bw_1(T,X)-\cdots-\bw_h(T,X)]/T^h$, where $\bw_0+T\bw_1(T,X)+\cdots$ is the formal expansion of $\bw$, and the $\bw_k$ ($k=0, 1, \dots$) are polynomials in $\ln T$. By substitution, one checks that $\mathbf z$ solves a system of the same form as the first reduced system, but with $A$ replaced by $A+h$. Taking $h$ large enough, one may always assume that $A+h$ is positive definite.
\end{remark}

\section{Step 3: Estimating $w$}

The local solution of the reduced system is obtained by a
modification of the method of solution of symmetric hyperbolic
systems. The net result, for our purposes, is that there is a
$T_0>0$ (possibly smaller than $b$ or $\alpha$), and a unique
solution $w$ for every choice of $w_0$ small in $H^s$, that is
continuous with values in $H^s$. Let us therefore fix some
constant $M > \ep$ such that, for $\|\psi\|_\sigma+\|w_0\|_s\leq
M$, we have $\|\bw(T)\|_s\leq 2M$ for $T\leq T_0$. By the
continuity of $w$, this is certainly true for small $T_0$. Since
$s>n/2$, this implies an $L^\infty$ bound as well. We may also
assume that $\psi$ is small enough in $H^\sigma$ to ensure that
$\gamma(=1-|\nabla\psi|^2)$ remains bounded away from zero. Since
the RS is a standard symmetric hyperbolic system for $T>0$, its
solutions persist as long as they do not blow up in $C^1$ (this
means that $u$ has no singularity other than the one for $T=0$).
Since $s>n/2+1$, this follows from a bound in $H^s\times H^{s-1}$.
We proceed to show that $\bw$ actually satisfies stronger
estimates that will enable us to show that it actually extends to
all $T\in[0,b]$, and remains small there.

For this, we must first estimate $f=Tf_0+T\ln Tf_1$. Recall that, by Moser-type estimates, the $H^s$ Sobolev norms of the powers of $\bw$ are estimated linearly in terms of the Sobolev norms of $\bw$ if $\bw$ is known to be bounded. By consideration of the expression for $f$, one obtains an estimate of the form
\[ \|f\|_s\leq K_1(M)T|\ln T|(1+T^8)(\|\psi\|_\sigma+\|\bw(T)\|_s).
\]
Similarly, we also have an $H^0$ (i.e., $L^2$) estimate
\[ \|f\|_0\leq K_2(M)T|\ln T|(1+T^8)(\|\psi\|_\sigma+\|\bw(T)\|_0).
\]
The derivation of the $L^2$ estimates on $\bw$ requires a slight modification of the standard scalar product. Let us write $(f,g)$ for the real $L^2$ scalar product on functions of the space variables $X$. Introduce the matrix $V=\mathop{{\rm diag}}(6\gamma,1,\dots,1)$; the matrix $VQA$ is then nonnegative, and satisfies $VA^i=A^i$. Multiplying (\ref{eq:symmetric}) by $V$ and taking the scalar product with $\bw$, we obtain, since $VQ$ is independent of $T$, $D(\bw,VQ\bw)=2(\bw,VQD\bw)$, hence
\[ \frac12 D(\bw,VQ\bw) + (\bw,VQA\bw) = \sum_iT(\bw,A^i\pa_i\bw)+(\bw,Vf).
\]
Now, $(\bw,VQA\bw)\geq 0$ and $(\bw,A^i\pa_i\bw)=(-\pa_i[A^i\bw],\bw)$ since $A^i$ is symmetric. It follows, by expanding $\pa_i[A^i\bw]$, that $2(\bw,A^i\pa_i\bw)=(-[\pa_i A^i]\bw,\bw)$. This quantity may be estimated by $C\|\pa_i A^i\|_{L^\infty}(\bw,V\bw)\leq C(M)(\bw,V\bw)$, since $A^i$ involves the first derivatives of $\psi$. Consider now $e_0(T)=(\bw,VQ\bw)(T)$, a quantity equivalent to the $L^2$ norm since $\gamma$ is bounded away from zero. We obtain
\[De_0\leq T|\ln T|(K_3(M)\|\psi\|_\sigma+K_4(M)e_0)(1+T^8),
\]
hence (remembering that $D=T\pa_T$),
\[\pa_T\left\{  \ln(K_3\|\psi\|_\sigma+K_4e_0(T)) \right\}\leq |\ln T|(1+T^{8}).
\]
Integrating, we obtain that for $0\leq T\leq b$, one has
\[ K_3\|\psi\|_\sigma+K_4e_0(T)\leq (K_3\|\psi\|_\sigma+K_4e_0(0))\exp\{ \int_0^b |\ln \tau|(1+\tau^8)d\tau \}.
\]
Since $e_0(T)$ is estimated by a multiple of $\|w_0\|_s$, we
obtain an inequality of the form
\[ e_0(T)\leq K_5\|\psi\|_\sigma + K_6\|w_0\|_s.
\]
To obtain spatial derivative norms, one performs the same work on the system solved by $\bv=S\bw$, estimating $e_s(T)=(S\bw,VS\bw)$ (this is equivalent to the norm of $\bw(T)$ in $H^{s}$). This system satisfies the same assumptions as the original one because of commutator estimates. Time derivatives are then estimated using the reduced system itself.

Take now $s$ to be an integer. If $\sigma$ is large enough at the
outset, there is a positive $\delta$ such that the inequality
$\|\psi\|_\sigma + \|w_0\|_s<\delta$ ensures that $w$ remains less
than $\ep$, hence less than $M$ up to time $T=b$ at least, and
therefore is well-defined on the slab $\alpha-1\leq T\leq b$.
Furthermore, by induction, $\pa_T^kw$ is, for $s-k>n/2$, bounded
in $H^{s-k}$ in this slab, so that $w$ belongs to the Sobolev
class $H^{m}$ \emph{in space and time} with respect to the $(X,T)$
variables if the integer $m$ satisfies $2m>s$ and $s-m>n/2$. For
$\sigma$, hence $s$ large enough, we may take $m$ greater than
both $n/2+1$ and $s_0+1$. Since the mapping
$(t,x)\mapsto(t-\psi(x),x)$ is a local diffeomorphism of class
$H^s$ for any $s>n/2+1$,\footnote{The composition of Sobolev maps
with $s>n/2+1$ is discussed, for instance, in \cite[p.~108]{e-m}.}
$w$ is also of class $H^m$ with respect to the original $(x,t)$
variables. This is also true of $T^3w$. The traces of this
function and its $t$-derivative on $(t=\alpha)$ therefore belong
in particular to $H^{s_0}$ and $H^{s_0-1}$ respectively, and are
small in these spaces if $\|\psi\|_\sigma + \|w_0\|_s$ is small
enough.

\section{Step 5: Estimating the Cauchy data of $u$}

At this stage, we know that the Cauchy data of $1/t$ and $\Phi$ on
$\{(x,t) : t=\alpha \text{ and }x\in\Omega\}$ are both less than
$\ep/4$ if $\psi$ is small enough. We also know that the Cauchy
data of $T^3w$ on $(t=\alpha)$ may be made less than $\ep/2$ in
$H^{s_0}\times H^{s_0-1}$ by choosing $\sigma$  and $s$ large
enough, and the singularity data $\psi$ and $w_0$ small enough in
their respective spaces. Therefore, the Cauchy data of
$u=1/t+\Phi+T^3w$ are less than $\ep$ in $H^{s_0}\times H^{s_0-1}$
if the singularity data are small enough, QED.

\section{Concluding remarks}

We have shown that the reachable set in the cubic nonlinear wave
equation contains solutions that blow up on any prescribed compact
set. The control time is here a priori very large if $\ep$ is
(that is, if the local controllability set is very small).
However, the solution is far from unique, since all solutions
having the same $\psi$ but different $w_0$ all have the same
blow-up set. This raises the question whether one may optimize the
control time by proper choice of $w_0$.\footnote{From a practical
standpoint, one suggestion would be to determine a formal solution
to high order, taking $w_0=\theta Z(X)$, with $Z$ a bump function
and $\theta$ small, and to plot the values of the norm of the
Cauchy data of the solution on some hyperplane $(t=\alpha')$, with
$\alpha'$ not too large, as a function of $\theta$. It is
conceivable that a value of $\theta$ yielding a minimum of this
norm would be a good candidate.} Also, since there are many
different functions $\psi$ that have the same zero set, the choice
of $\psi$ could also be of some interest, given that the curvature
of the blow-up set is related to the rate of concentration of the
so-called ``energy'' \cite{SK-laser}.\footnote{Thus, for blow-up
at a single point, a sharply peaked $\psi$ could be preferable in
some applications, and in others a shallow extremum would on the
contrary be desirable.}

The possibility of control of singularities seems to be a further
illustration of Russell's suggestion \cite[pp.~640-641]{russell}
that the development of control theory was slowed down by the
``historically dominant emphasis on well-posedness and
regularity'' in the study of PDEs, putting to the fore the search
for conditions under which the influence of the data on the
solution is ``not too great.'' By contrast, in control theory ``we
want to know that the influence of the control functions on [the
solution] is `not too little'---''and the present paper shows that
this influence may be largest possible, and force the solution to
become infinite. The main technical point is that the solution
admits a stable parameterization of solutions by singularity
data---and not only by Cauchy data: blow-up is a stable
phenomenon, amenable to control.

\medskip\quad
Received October 2012; revised August 2013.
\medskip

\end{document}